\documentclass[11pt]{amsart}
\usepackage[parfill]{parskip}    
\usepackage{fullpage}
\usepackage{graphicx}
\usepackage{amssymb,amsmath,latexsym}
\usepackage{tikz}
\usetikzlibrary{arrows,decorations.markings,decorations.pathreplacing,calc,matrix,intersections}
\tikzset{->-/.style={decoration={markings,mark=at position #1 with {\arrow{>}}},postaction={decorate}}}

\usepackage{epstopdf}
\usepackage{ifpdf}
\usepackage{color}
\usepackage{float}
\usepackage{amscd,amsmath, mathabx}
\usepackage{amssymb,amsmath,latexsym,color,enumerate,tikz}
\usepackage[all]{xypic}       
\usepackage[varg]{pxfonts}
\usepackage{amsmath,amsthm,amsfonts,amssymb,fancyhdr,graphics,relsize,tikz-cd,mathtools,faktor,tikz,changepage}
\usepackage{graphicx}
\usepackage{placeins}

\usepackage{dsfont}
\definecolor{red}{rgb}{1,0,0} 

 \definecolor{darkgreen}{rgb}{0, .7, 0}

 \definecolor{purple}{rgb}{.7, 0, 1}

\tikzset{mynode/.style={draw,circle,fill=black,inner sep=2pt,outer sep=0.5pt}}

\newtheorem{theorem}{Theorem}[section]
\newtheorem*{theorem*}{Theorem}
\newtheorem*{lemma*}{Lemma}
\newtheorem{proposition}[theorem]{Proposition}
\newtheorem{lemma}[theorem]{Lemma}
\newtheorem{corollary}[theorem]{Corollary}

\theoremstyle{definition}

\theoremstyle{remark}

\usepackage{hyperref}
 
\begin{document}
\title{Pro-$p$ completions of $PD_n$-groups}
\author{Jonathan Hillman, Dessislava H. Kochloukova}

\address{School of Mathematics and Statistics,
University of Sydney,
Sydney, NSW,
Australia}
\address
{Department of Mathematics, State University of Campinas (UNICAMP), 13083-859, Campinas, SP, Brazil }
\email{jonathanhillman47@gmail.com}
\email{desi@ime.unicamp.br}


\begin{abstract} 
We sharpen earlier work on the pro-$p$ completions of orientable 
$PD_3$-groups.
There are four cases, 
and we give examples of aspherical 3-manifolds representing each case.
In  three of the four cases the new results are best possible.
We also consider the pro-$p$ completion of some orientable $PD_n$ groups for $n \leq 5$,
 including surface-by-surface groups.

\end{abstract}

\maketitle

\section{Introduction}

There are two definitions of a profinite Poincaré duality group $G$ 
of dimension $n$ at a prime $p$ \cite[3.4.6]{book1}, \cite{W-S}. 
Both definitions differ on whether the profinite group $G$ should be
of type $FP_{\infty}$ over $\mathbb{Z}_p$ i.e. whether the trivial $\mathbb{Z}_p[[G]]$-module $\mathbb{Z}_p$ has a projective resolution with 
all projectives finitely generated $\mathbb{Z}_p[[G]]$-modules. 
The groups that satisfy the definition of \cite{W-S} we call strong profinite $PD_n$ groups at $p$ and the groups that satisfy the original definition of Tate from \cite{book1}, \cite{Serre} we call profinite $PD_n$ groups at $p$. 
By \cite{W-S} every strong profinite $PD_n$ group at $p$ is a profinite $PD_n$ group at $p$. 
By definition a group $G$ which is a strong $PD_n$ group at $p$ has cohomological $p$-dimension $cd_p(G) = n$, has type $FP_{\infty}$ over 
$\mathbb{Z}_p$  and $Ext^i_{\mathbb{Z}_p[[G]]}(\mathbb{Z}_p, \mathbb{Z}_p[[G]]) = 0$ for $i \not= n$ and 
$Ext^n_{\mathbb{Z}_p[[G]]}(G, \mathbb{Z}_p[[G]]) \simeq \mathbb{Z}_p$. 
If the action of $G$ on  $Ext^n_{\mathbb{Z}_p[[G]]}(G, \mathbb{Z}_p[[G]]) \simeq \mathbb{Z}_p$ is trivial $G$ is called orientable.
For pro-$p$ groups the notions of strong profinite $PD_n$ group at $p$ and  profinite $PD_n$ group at $p$ coincide. 
We call such groups pro-$p$ $PD_n$ groups.

We are interested in pro-$p$ and profinite completions of orientable $PD_n$ groups.
The cases $n=1$ or 2 are well understood.
The pro-$p$ completions of orientable $PD_2$-groups are
pro-$p$ $PD_2$-groups.
These are also known as Demu\v skin groups,
and were completely classified in terms of pro-$p$ generators and relations in \cite{D1}, \cite{D2}, \cite{Labute}, \cite{S}.
(Not all such groups are pro-$p$ completions of $PD_2$-groups.)
Profinite and pro-$p$ completions of $PD_3$ groups were studied by Kochloukova and Zalesski in \cite{Desi-Pavel} and by Weigel in \cite{Wilkes}. Some results on pro-$p$ completions for arbitrary $n$ were obtained by Hillman, Kochloukova and Lima in \cite{H-K-L}. 
The notion of orientable profinite Poincaré duality pairs (over $\mathbb{F}_p$) was first suggested by Kochloukova in \cite{Koch} and a more general notion of (in general non-orientable) profinite Poincaré duality pairs was developed by Wilkes in \cite{Wilkes}. 

Sections 2 and 3 contain some basic definitions, 
lemmas and results from earlier work.
In Section \ref{proPD3} we build upon the results of \cite{Desi-Pavel},
and prove the following theorem.

\medskip
{\bf Theorem A.} {\it 
 Let $G$ be an orientable Poincaré duality group of dimension 3 and let $\widehat{G}_p$ be the pro-$p$ completion of $G$. 
Then exactly one of the following conditions holds:

a) $\widehat{G}_p$ is cyclic or quaternionic;

b)  $\widehat{G}_p$ is an orientable pro-$p$ $PD_3$-group;

c)  there is no upper bound on the deficiency of the subgroups of finite index in $\widehat{G}_p$; 

d) $\widehat{G}_p$ is  $\mathbb{Z}_p$ or $\widehat{D_\infty}_2$.}

The statements of cases (a) and (d)  sharpen the corresponding statements
in  \cite[Thm B]{Desi-Pavel}.
We also give simple criteria for when they arise.
Here our results are essentially complete (except for  $p=2$).
Several equivalent criteria for case (b) were given in \cite[Thm A]{Desi-Pavel}.
We augment these criteria, 
as a corollary of Theorem \ref{cdleq2 implies free}.
This theorem also implies that $\widehat{G}_p$ cannot
have cohomological $p$-dimension 2.
However it is not yet clear what else might occur,
and we do not have simple criteria for recognizing case (c).

In \S\ref{examples} we give examples of geometric flavour for each 
of the four cases listed above.
We give an example of case (c) in which $\widehat{G}_p$ 
is a free pro-$p$ group of rank 2,
for all primes $p$. 
(We do not know whether there are examples of case (c)
in which $\widehat{G}_p$ has cohomological $p$-dimension
at least 3.)

In \cite{Serre} Serre called an abstract group $G$ good if for every finite 
$G$-module $M$ the map $H^i (\widehat{G}, M) \to H^i(G, M)$, 
induced by the canonical map $G \to \widehat{G}$, is an isomorphism,
 where $\widehat{G}$ denotes the profinite completion of $G$.
The group $B$ is called $p$-good if for every finite pro-$p$ 
$\mathbb{Z}_p[[\widehat{B}_p]]$-module $M$ we have that 
the canonical map $B \to \widehat{B}_p$ induces an isomorphism 
$H^i(\widehat{B}_p, M) \to H^i(B, M)$, 
where  $\widehat{B}_p$ is the pro-$p$ completion of $B$.  
Groups that are  $p$-good were previously studied in \cite{Koch2}, \cite{Koch}, \cite{Desi-Pavel}. 
In \cite{G-J-P-Z} the term $p$-good group was used with a different 
(but related)  meaning.

For a set $\mathcal{T}$ of  normal subgroups of $p$-power index in a discrete group $B$ we say that $\mathcal{T}$ is directed if for every $U_1, U_2 \in \mathcal{T}$ we have that there is  
$U \subseteq U_1 \cap U_2$ such that $U \in \mathcal{T}$.
In the next section we give criteria for groups of type $FP_\infty$
and with additional structure to be good, or $p$-good,
and we prove the following theorem.

\medskip
{\bf Theorem B.} {\it Let $1 \to A \to B \to C \to 1$ be a short exact sequence 
of groups such that $A$ is an orientable surface group and $C$ 
is an orientable $PD_s$ group, where $s = 2$ or $s = 3$.  
Let $\mathcal{T}$ be a directed set of  normal subgroups of $p$-power index 
in $B$ that defines the pro-$p$ topology of $B$. 
Suppose that 
$ \underset{ U \in {\mathcal T}}{\varprojlim} H_1(U \cap A, \mathbb{F}_p) = 0$ and furthermore if $s = 3$, there is an upper bound on the deficiency of the subgroups of finite index in $\widehat{C}_p$. 
Then 
    
a) if $s = 2$, then  $\widehat{B}_p$ is a pro-$p$ $PD_4$-group 
and $B$ is $p$-good. 
    
b) if $s = 3$,  then $\widehat{B}_p$ is virtually a pro-$p$ $PD_k$-group 
for some $k \in \{ 2,3, 5\}$. If $k = 5$ then $B$ and $C$ are $p$-good.
    
If additionally $B$ is orientable and $p$-good then $\widehat{B}_p$ is an orientable pro-$p$ $PD_{2+ s}$-group.}

\medskip
The case of profinite completions of orientable $PD_4$-groups is easier 
and is considered in Proposition \ref{profinite}.

It is an open problem whether there is an orientable $PD_n$-group $G$ such that $\widehat{G}_p$ is an orientable pro-$p$ $PD_n$-group
and $G$ is not $p$-good.

The main result of Section \ref{PD4-0} is the following theorem.

\medskip
{\bf Theorem C.} {\it Let $1 \to A \to B \to C \to 1$ be a short exact sequence of groups such that $A \simeq \mathbb{Z}^2$, 
$B$ is an orientable $PD_4$ group and $C$ is an orientable surface group.
Then one of the following holds:

a) $\widehat{B}_p$ is an orientable pro-$p$ $PD_4$-group and $B$ is 
$p$-good;

b) $\widehat{B}_p$ is an orientable pro-$p$ $PD_2$ group and the image of $A$ in $\widehat{B}_p$ is trivial.
}

\medskip
{\it Remark.} If $B$ is not orientable, $p \not= 2$, then there is a third option for the closure $\overline{A}$ of the image of $A$ in $\widehat{B}_p$ to be virtually $\mathbb{Z}_p$. 

We show also that if $B$ is an orientable $PD_4$ group and $\chi(B)=0$ 
then $\widehat{B}_p$ cannot be a pro-$p$ $PD_3$-group,
and we give examples of orientable
$PD_4$-groups which are fundamental groups of bundles with 
base and fibre aspherical closed surfaces, 
and for which the projection to the base induces an isomorphism 
on pro-$p$ completions, for all primes $p$.

In \cite{H-K-L} it was shown that under some conditions the pro-$p$ 
completion of an orientable $PD_n$ group is virtually a pro-$p$ 
$PD_r$-group,  for $r\leq n, r\not= n-1$.
In the final Section \ref{more} we give an example of an aspherical 
5-manifold with perfect fundamental group,
which completes the discussion of examples 
with ``dimension drop" $n-r\not=1$ in \cite{H-K-L}.
We do not know of any orientable $PD_n$-group $G$ whose 
pro-$p$ completion $\widehat{G}_p$ is virtually a pro-$p$ $PD_{n-1}$-group. 
Note that if $\widehat{G}_p$ is virtually a pro-$p$ $PD_{n-1}$-group
then $G$ has a subgroup of $p$-power index $H$ 
such that $\widehat{H}_p$ is a pro-$p$ $PD_{n-1}$-group and by Theorem \ref{cdleq2 implies free} this is impossible for $n = 3$.

{\bf Acknowledgments} The second named author was partially supported by Bolsa de produtividade em pesquisa CNPq 305457/2021-7 and Projeto tem\'atico FAPESP 18/23690-6.

\section{Preliminaries} 

Let $G$ be a group, and let $\{\gamma_iG\}$ be its lower central series, 
with $\gamma_1G=G$ and $\gamma_{i+1}G=[G,\gamma_iG]$ for all $i\geq1$.
If $p$ is a prime, let $X^p(G)$ be the subgroup generated by all $p$th powers
of members of $G$.
Let $D_\infty=\mathbb{Z}\rtimes(\mathbb{Z}/2\mathbb{Z})$ be the infinite dihedral group.

Let $G$ be a profinite group. 
By definition $\mathbb{Z}_p[[G]] = {\varprojlim} \frac{\mathbb{Z}}{p^i \mathbb{Z}} [[G/ U]],$ where the inverse limit is over all $i \geq 1$ and $U$ open subgroups of $G$. And $\mathbb{F}_p[[G]] =\mathbb{Z}_p[[G]]/ p \mathbb{Z}_p[[G]] =  {\varprojlim} \mathbb{F}_p[[G/U]]$
where the inverse limit is over all open subgroups $U$ of $G$.

When $G$ is a group, $H_i(G, V )$ denotes the $i$th homology of $G$ in the respective category.
Thus if $G$ is an abstract group $V$ is a $\mathbb{Z} G$-module, 
if $G$ is a pro-$p$ group $V$ is a pro-$p$ $\mathbb{Z}_p[[G]]$-module 
and if $G$ is a profinite group $V$ is a  profinite 
$\widehat{\mathbb{Z}}[[G]]$-module. 
Furthermore $H^i(G, W )$ denotes the $i$th cohomology of $G$ in the respective category.
If $G$ is an abstract group $W$ is a $\mathbb{Z} G$-module, 
if $G$ is a pro-$p$ group or more generally a profinite group $W$ is a 
discrete $G$-module and so $W = \cup W^{U}$ where the union is 
over all open subgroups $U$ of $G$. 
In our applications $V$ and $W$ will be finite.

Since the pro-$P$ completions of $\mathbb{Z}$ and of surface groups ($PD_2$-groups) are well understood,
the first interesting case is in dimension 3.

\begin{theorem} \cite[Thm B]{Desi-Pavel}
 \label{PD3-} Let $G$ be an orientable Poincaré duality group of dimension 3 and let $\widehat{G}_p$ be the pro-$p$ completion of $G$. 
Then exactly one of the following conditions holds

a) $\widehat{G}_p$ is finite;

b) $\widehat{G}_p$  is an orientable pro-$p$ $PD_3$-group;

c)  there is no upper bound on the deficiency of the subgroups of finite index in $\widehat{G}_p$; 

d) $\widehat{G}_p$ is virtually $\mathbb{Z}_p$.
\end{theorem}

By the proof of Theorem \ref{PD3-} if $ \underset{ U \in {\mathcal T}}{\varprojlim} H_2(U, \mathbb{F}_p) = 0$ then case b) from Theorem \ref{PD3-} holds. Furthermore if  a), b), c) do not hold ( and so d) holds) then 
$ \underset{ U \in {\mathcal T}}{\varprojlim} H_2(U, \mathbb{F}_p)  \simeq \mathbb{F}_p$.

\begin{theorem}  \cite[Thm. 4]{Koch2} \label{PD-PD} Let $G$ be an abstract Poincaré duality group of dimension $m$ and let $\mathcal{C}$ be a directed set of normal subgroups of finite index in $G$. Suppose further that there is a subgroup $G_0$ of finite index in $G$ such that $G_0$ is orientable, that there is some $U_0 \in \mathcal{C}$ such that $U_0 \subseteq G_0$ and that, for all $i \geq 1$, 
 $$  \underset{ U \in {\mathcal C}}{\varprojlim} \  H_i(U, \mathbb{F}_p) = 0.$$

Then $\widehat{G}_{\mathcal{C}}$ is a strong profinite Poincaré duality group of dimension $m$ at $p$,  $\widehat{(G_0)}_{\mathcal{C}}$ is a strong profinite Poincaré duality group of dimension $m$ at $p$ and  $\chi_p(\widehat{G}_{\mathcal{C}}) = \chi(G)$.
\end{theorem}

\section{Some auxiliary results} 
 
 We will need the following simple lemmas.
 
 \begin{lemma} \label{H2} 
Let $S$ be an orientable $PD_2$-group with a subnormal subgroup $D$ of index $p^k$, where $p$ is prime, and let $j:D\to{S}$ be the inclusion.
Then $H_2(j;\mathbb{F}_p)=0$ and $H^2(j;\mathbb{F}_p)=0$.
\end{lemma}

\begin{proof}
Since $D$ is subnormal there is a chain $D=D_1<\dots<D_m=S$,
where $D_i$ is normal in $D_{i+1}$ and $[D_{i+1}:D_i]=p$,
for all $i<m$.
It shall suffice to show that 
$H_2(D_i;\mathbb{F}_p)\to{H_2(D_{i+1};\mathbb{F}_p)}$ is the zero map.
Thus we can assume that $D$ is a normal subgroup of $S$ of index $p$.

Let $j_*=H_*(j;\mathbb{F}_p)$ and $j^*=H^*(j;\mathbb{F}_p)$, 
for simplicity of notation.
Let $x\in{H^1(S;\mathbb{F}_p)=Hom(S,\mathbb{Z}/p\mathbb{Z})}$ 
be an epimorphism with kernel $D$.
Since $S$ is an orientable $PD_2$-group and $x\not=0$ there is a 
$y\in{H^1(S;\mathbb{F}_p)}$ such that $x\cup{y}$ generates
$H^2(S;\mathbb{F}_p)$.
If we evaluate $x\cup{y}$ on the image of a class 
$\delta\in{H_2(D;\mathbb{F}_p)}$ we get
$(x\cup{y})(j_2\delta)=j^*(x\cup{y})(\delta)=(j^*x\cup{j^*y})(\delta)=0$,
since $j^*x=0$ is the restriction of $x$ to $D$.
Hence $j_2\delta=0$, for all $\delta$, and so 
$H_2(j;\mathbb{F}_p)=0$.
The dual result $H^2(j;\mathbb{F}_p)=0$ follows immediately.
\end{proof}

 For an abstract  group $U$ denote by $\widehat{U}_p$ the pro-$p$ completion of $U$.
 
 \begin{lemma} \label{inv-inv} Let $G$ be an abstract group, $\mathcal{M}$ be a  directed set of normal subgroups of $p$-power index in $G$ that define the pro-$p$ topology on $G$ . Then 
 $ \underset{ M \in {\mathcal M}}{\varprojlim}  H_1(M, \mathbb{F}_p) = 0.$
 \end{lemma} 
 
 \begin{proof} Let $\overline{M}$ be the closure of $M \in \mathcal{M} $ in $\widehat{G}_p$. Then since $\overline{M} \simeq \widehat{M}_p$ and $\bigcap_{M \in {\mathcal{M}}} \overline{M} = 1$ we have
 $$ \underset{ M \in {\mathcal M}}{\varprojlim}  H_1(M, \mathbb{F}_p) \simeq 
  \underset{ M \in {\mathcal M}}{\varprojlim}  H_1(\overline{M}, \mathbb{F}_p) \simeq
   H_1(  \underset{ M \in {\mathcal M}}{\varprojlim} \overline{M}, \mathbb{F}_p) = H_1( \bigcap_{M \in {\mathcal{M}}} \overline{ M}, \mathbb{F}_p) = 0.$$
   \end{proof} 
   
\begin{proposition} \label{inverse-prop} 
Let $1 \to A \to B \to C \to 1$ be a short exact sequence of abstract groups. Let $\mathcal{T}$ be a directed set of subgroups in $B$. 
Suppose that  each $H_j(U \cap A, \mathbb{F}_p)$ is finite and 
    $  \underset{ U \in {\mathcal T}}{\varprojlim} H_j(U \cap A, \mathbb{F}_p) = 0$. Then
    $  \underset{ U \in {\mathcal T}}{\varprojlim} H_i(U/ (U \cap A), H_j(U \cap A, \mathbb{F}_p)) = 0$ for $i \geq 0$.
   \end{proposition}
   
   \begin{proof} Set $M_U =  H_j(U \cap A, \mathbb{F}_p)$ and $V_U = U/ (U \cap A)$ a subgroup of $C$. Let $${\mathcal R} : \ldots \to R_i \to R_{i-1} \to \ldots \to R_0 \to \mathbb{Z} \to 0$$ be a free resolution of the trivial $\mathbb{Z} C$-module $\mathbb{Z}$. Then $H_i(V_U, M_U) = H_i (\mathcal{R}_U)$, where ${\mathcal{R}}_U = \mathcal{R} \otimes_{V_U} M_U$. The maps of the inverse system $\{  H_i(V_U, M_U) \ | \  U \in {\mathcal{T}} \}$  can be described as follows:  if $U_1, U_2 \in {\mathcal{T}}$, where $U_1 \subseteq U_2$ the map $\varphi_{U_1, U_2} : H_i(V_{U_1}, M_{U_1}) \to H_i(V_{U_2}, M_{U_2})$ is induced by the map $id_{\mathcal{R}} \otimes d_{U_1, U_2} : \mathcal{R} \otimes_{V_{U_1}} M_{U_1} \to \mathcal{R} \otimes_{V_{U_2}} M_{U_2}$ that sends $r_i \otimes m$ to $r_i \otimes d_{U_1, U_2}(m)$ for $r_i \in R_i$  and $ d_{U_1, U_2} : M_{U_1} \to M_{U_2}$ is induced by the inclusion map $U_1 \cap A \to U_2 \cap A$.
   
   Since $  \underset{ U \in {\mathcal T}}{\varprojlim} M_U = 0$ and each $M_U$ is finite, then for every $U_2 \in {\mathcal{T}}$ there is $U_1$ as above such that $d_{U_1, U_2}$ is the zero map. Then
   $\varphi_{U_1, U_2}$ is the zero map and hence  $  \underset{ U \in {\mathcal T}}{\varprojlim}  H_i(V_{U}, M_{U}) = 0$.
   \end{proof}
   
   \begin{lemma} \label{pro-p} Let $A$ be an orientable surface group. Let ${\mathcal{S}}$ be a directed set of normal subgroups of $p$-power index in $A$. Suppose that   $  \underset{ U \in {\mathcal S}}{\varprojlim} H_1(U, \mathbb{F}_p) = 0$. Then the completion $\overline{A} =  \underset{ U \in {\mathcal S}}{\varprojlim} A/ U$ of $A$ with respect to ${\mathcal{S}}$ is isomorphic to the pro-$p$ completion $\widehat{A}_p$.
   \end{lemma}
   
  \begin{proof}
   Consider the cellular chain complex associated to the standard Cayley complex of $A$, i.e.
   $${\mathcal{R}} : 0 \to \mathbb{Z} A \to (\mathbb{Z}A )^d \to \mathbb{Z} A \to \mathbb{Z} \to 0$$
   Consider the complexes  $\overline{\mathcal{R}} = \mathbb{F}_p[[\overline{A}]] \otimes_{\mathbb{Z} A} {\mathcal{R}}$ and  $\widehat{\mathcal{R}} = \mathbb{F}_p[[\widehat{A}_p]] \otimes_{\mathbb{Z} A} {\mathcal{R}}$.
   By \cite[Lemma 2.1]{Desi-Pavel}
   $H_i (\overline{\mathcal{R}}) =  \underset{ U \in {\mathcal S}}{\varprojlim}\  H_i(U, \mathbb{F}_p)$.
   Thus $H_1 (\overline{\mathcal{R}}) = 0$.
      
 Note that $\widehat{\mathcal{R}}$  is a free resolution of the trivial $\mathbb{F}_p[[\widehat{A}_p]]$-module $\mathbb{F}_p$. Let $\mathcal{T}$ be the directed set of all  normal subgroups of $p$-power index in $A$.
      Let $K = Ker (\widehat{A}_p \to \overline{A})$ and $\widehat{A}_p = \underset{ U \in {\mathcal T}}{\varprojlim} A/ U  \to \overline{A} = \underset{ U \in {\mathcal S}}{\varprojlim} A/ U$ is the epimorphism induced by the identity maps $id_{A/U }$ for $U \in {\mathcal S} \subseteq {\mathcal T}$.  Thus $\overline{\mathcal{R}} \simeq \mathbb{F}_p \otimes_{\mathbb{F}_p[[K]]} \widehat{\mathcal{R}}$. Then
   $$H_1(K, \mathbb{F}_p) = H_1( \mathbb{F}_p \otimes_{\mathbb{F}_p[[K]]} \widehat{\mathcal{R}}) \simeq H_1 (\overline{\mathcal{R}}) = 0 $$
Hence $K = 1$ and $\widehat{A}_p \simeq \overline{A}$.   
 \end{proof}

\begin{lemma} \label{obvious} Let $G$ be a group with pro-$p$ completion $\widehat{G}_p$.
Denote $\mu^i : H^i(\widehat{G}_p, \mathbb{F}_p) \to H^i(G, \mathbb{F}_p)$ the map induced by  the canonical map $G \to \widehat{G}_p$. Then we have a commutative diagram
\[ \begin{tikzcd}
H^i(\widehat{G}_p, \mathbb{F}_p) \times H^j(\widehat{G}_p, \mathbb{F}_p)  \arrow{r}{\cup} \arrow[swap]{d}{\mu^i \times \mu^j} & H^{i+j}(\widehat{G}_p, \mathbb{F}_p)  \arrow{d}{\mu^{i+ j}} \\
H^i({G}, \mathbb{F}_p) \times H^j({G}, \mathbb{F}_p) \arrow{r}{\cup}& H^{i+j}({G}, \mathbb{F}_p)
\end{tikzcd}
\]
where the horizontal maps are the cup products in the categories of pro-$p$ 
and abstract groups.
\end{lemma}

\begin{proof} Following \cite{Serre} consider the set ${C}^n(\widehat{G}_p, \mathbb{F}_p)$ of all continuous maps $\widehat{G}_p^n \to \mathbb{F}_p$.  Then there is a map
$\cup :  {C}^i(\widehat{G}_p, \mathbb{F}_p) \times {C}^j(\widehat{G}_p, \mathbb{F}_p) \to {C}^{i+j}(\widehat{G}_p, \mathbb{F}_p)$ defined by $(f \cup h) (g_1, \ldots, g_{i+j}) = f(g_1, \ldots, g_i) h(g_{i+1}, \ldots, g_{i+j})$ that induces the cup product $\cup :  {H}^i(\widehat{G}_p, \mathbb{F}_p) \times {H}^j(\widehat{G}_p, \mathbb{F}_p) \to {H}^{i+j}(\widehat{G}_p, \mathbb{F}_p)$.

Similarly we can consider the set ${C}_0^n({G}, \mathbb{F}_p)$ of all maps ${G}^n \to \mathbb{F}_p$.  Then there is a map
$\cup :  {C}_0^i({G}, \mathbb{F}_p) \times {C}_0^j({G}, \mathbb{F}_p) \to {C}_0^{i+j}({G}, \mathbb{F}_p)$ defined by $(f \cup h) (g_1, \ldots, g_{i+j}) = f(g_1, \ldots, g_i) h(g_{i+1}, \ldots, g_{i+j})$ that induces the cup product $\cup :  {H}^i({G}, \mathbb{F}_p) \times {H}^j({G}, \mathbb{F}_p) \to {H}^{i+j}({G}, \mathbb{F}_p)$.

The canonical map $G \to \widehat{G}_p$ induces maps $\nu^i: C^i(\widehat{G}_p, \mathbb{F}_p) \to C^i_0(G, \mathbb{F}_p)$, that induce the maps $\mu^i$. Then by the definition of the cup product we have a commutative diagram
\[ \begin{tikzcd}
C^i(\widehat{G}_p, \mathbb{F}_p) \times C^j(\widehat{G}_p, \mathbb{F}_p)  \arrow{r}{\cup} \arrow[swap]{d}{\nu^i \times \nu^j} & C^{i+j}(\widehat{G}_p, \mathbb{F}_p)  \arrow{d}{\nu^{i+ j}} \\
C_0^i({G}, \mathbb{F}_p) \times C_0^j({G}, \mathbb{F}_p) \arrow{r}{\cup}& C_0^{i+j}({G}, \mathbb{F}_p)
\end{tikzcd}
\]
and this commutative diagram induces the commutative diagram from the statement of the lemma.
\end{proof}

\section{pro-$p$ completions of $PD_3$-groups}
\label{proPD3}

In this section we shall sharpen some of the results of \cite{Desi-Pavel}.
Cases (a), (b), (c) and (d) shall refer to the four possibilities 
in the statement of Theorem \ref{PD3-}.

We begin by refining the statements of cases (a) and (d).
Theorem A of the introduction is then an immediate consequence.

\begin{lemma}
\label{virtually Zp}
Let $G$ be a finitely generated group and $p$ be a prime,
and let $K$ be the kernel of the natural homomorphism from $G$ to 
$\widehat{G}_p$.
Suppose that $\widehat{G}_p$ is virtually $\widehat{\mathbb{Z}}_p$.
Then $G/K$ has a finite normal subgroup $F$ which is a $p$-group,
and such that $G/K\cong{F}\rtimes\mathbb{Z}$ if $p$ is odd, 
while the quotient of $G/K$ by $F$ is $\mathbb{Z}$ or $D_\infty$ if $p=2$.
\end{lemma}

\begin{proof}
Since $\widehat{G}_p$ is virtually $\widehat{\mathbb{Z}}_p$,
there is a short exact sequence 
\[
0\to\widehat{\mathbb{Z}}_p\to\widehat{G}_p\to{T}\to0,
\]
where $T$ is a finite $p$-group.
Hence there is a short exact sequence $ 1 \to A \to G/ K \to T \to 1$, 
where $A\cong\mathbb{Z}$.
Therefore $ G/K$ has two ends, 
and so it has a maximal finite normal subgroup $F$ 
with quotient $\mathbb{Z}$ or $D_\infty$.
The subgroup $F$ maps injectively to $T$, and so is a $p$-group.
If $p$ is odd then $A$ is central, since $Aut(A)=\{\pm1\}$ has order 2.
Since $[G/K:A]$ is finite,  $(G/K)'$ is finite, by a lemma of Schur.
Hence $G/K\cong{F}\rtimes\mathbb{Z}$, with $F$ finite.
\end{proof}

\begin{theorem}
\label{refining 2.1}
Let  $G$ be an orientable $PD_3$-group and $p$ be a prime. Then
\begin{enumerate}
\item{}if $\widehat{G}_p$ is finite then it is cyclic or quaternionic;
\item{}if $\widehat{G}_p$ is virtually $\widehat{\mathbb{Z}}_p$ then 
either $\widehat{G}_p\cong\widehat{\mathbb{Z}}_p$ or $p=2$ 
and $\widehat{G}_2\cong\widehat{D_\infty}_2$.
\end{enumerate}
\end{theorem}

\begin{proof}
Let $K$ be the kernel of the natural homomorphism from $G$
to $\widehat{G}_p$. 
Suppose first that $\widehat{G}_p$ is finite.
Then $P=G/K\cong\widehat{G}_p$ is a finite $p$-group, 
and $K$ has no quotient which is a finite $p$-group.
Hence $H_1(K;\mathbb{Z})$ is  finite, 
of order prime to $p$, and $H^1(K;\mathbb{Z})=0$.
Consider the LHS spectral sequence 
\[
E^{p,q}_2=H^p(P;H^q(K;\mathbb{Z}))\Rightarrow{H^{p+q}(G;\mathbb{Z})}
\]
for the cohomology of $G$.
The group $P$ acts trivially on $H^0(K;\mathbb{Z})$ and $H^3(K;\mathbb{Z})$,
since $G$ is orientable, while $H^1(K;\mathbb{Z})=0$ and 
$H^2(K;\mathbb{Z})\cong{H_1(K;\mathbb{Z})}$.
Therefore $E^{p,q}_2=0$ if $p>0$ and $q\not=0$ or 3, 
or if $p=0$ and $q\not=0$, 2 or 3.

A slight extension of the argument of \cite[Lemma IV.6.2]{Adem-Milgram}
shows that $P$ has periodic cohomology
(with period dividing 4).
Since $P$ is a $p$-group, it must be either cyclic (of prime power order)
or quaternionic (of order a power of 2).
(Note that a finite group has periodic cohomology if and 
only if every abelian subgroup is cyclic.)

Suppose now that  $\widehat{G}_p$ is virtually $\widehat{\mathbb{Z}}_p$.
It follows from Lemma \ref{virtually Zp}
that $G$ has normal subgroups $K<L$ 
such that $G/L\cong\mathbb{Z}$ or $D_\infty$
and $J=L/K$ is a finite $p$-group, 
while $K$ has no non-trivial quotient which is a $p$-group.
On passing to a subgroup of index 2 in $G$, if necessary, 
we may assume that $G/L\cong\mathbb{Z}$.

Let $\Lambda=\mathbb{F}_p[G/K]$.
Then $H_i(K;\mathbb{F}_p)=H_i(G;\Lambda)$, for all $i$.
Clearly $H_0(K;\mathbb{F}_p)=\mathbb{F}_p$ and 
$H_1(K;\mathbb{F}_p)=0$,
while $H_i(K;\mathbb{F}_p)=0$ for $i>2$,
since $K$ has infinite index in $G$ and so $cd(K) < 3$ \cite{Strebel}.
We also have
$H_2(K;\mathbb{F}_p)=H_2(G;\Lambda)\cong{H^1(G;\Lambda)}$,
by Shapiro's Lemma and Poincaré duality.
Now $H_0(G;\Lambda)=\mathbb{F}_p$ and 
$H_1(G;\Lambda)=H_1(K;\mathbb{F}_p)=0$.
Hence $H^1(G;\Lambda)\cong{Ext^1_\Lambda(\mathbb{F}_p,\Lambda)}$,
by the Universal Coefficient spectral sequence
(or by an {\it ad hoc\/}low-degree argument).
Since $G/K$ has two ends, 
$Ext^1_\Lambda(\mathbb{F}_p,\Lambda)\cong\mathbb{F}_p$.
Thus we conclude that
$H_2(K;\mathbb{F}_p)\cong  H^1(G;\Lambda) \cong \mathbb{F}_p$.

The  LHS spectral sequence for the $\mathbb{F}_p$-cohomology 
of $L$ associated to the extension $1\to{K}\to{L}\to{T}\to1$ may be identified with the Leray-Serre spectral sequence for the fibration of $K(L,1)$ 
over $K(T,1)$.
The fibre $K(K,1)$ is  a $\mathbb{F}_p$-homology 2-sphere.
Since $T$ is a $p$-group it acts trivially on $\mathbb{F}_p$,
and therefore acts trivially on $H_*(K;\mathbb{F}_p)$.
Hence this spectral sequence reduces to a Gysin sequence
\[
\dots\to{H^{k+2}(L;\mathbb{F}_p)}\to{H^k(T;\mathbb{F}_p)}\to{H^{k+3}(T;\mathbb{F}_p)}\to{H^{k+3}(L;\mathbb{F}_p)}\to\dots,
\]
where the middle homomorphism is given by cup-product with a class
$z\in{H^3(T;\mathbb{F}_p)}$,
as in \cite[Example 5.C]{McCleary}.
Since $H^i(L;\mathbb{F}_p)=0$ for $i\geq3$ it follows that
these cup products induce isomorphisms
$H^i(T;\mathbb{F}_p)\to{H^{i+3}(T;\mathbb{F}_p)}$, for all $i\geq0$. 
Hence $T$ has cohomological period (dividing) 3.
But  a non-trivial finite group with periodic cohomology has even 
cohomological period \cite[Exercise VI.9.1]{Brown-book}.
Hence $T$ must be trivial.
\end{proof}

Theorem A now follows immediately from Theorems \ref{PD3-} and 
\ref{refining 2.1}.

We may easily identify the orientable $PD_3$-groups 
with pro-$p$ completion of type (a) or (d), 
when $p$ is odd.
(We do not yet have a comparably simple characterization when $p=2$.)

\begin{corollary}
If $p$ is an odd prime then $\widehat{G}_p$ is finite if and only if
$G/G'$ is finite and has cyclic $p$-torsion,
while $\widehat{G}_p\cong\widehat{\mathbb{Z}}_p$ if and only if
$G/G'\cong\mathbb{Z}\oplus{T}$,
where $T$ is finite and $(p,|T|)=1$.
\qed
\end{corollary} 

The criteria for recognizing when case (b) or (c) occurs are less complete.

\begin{theorem}
\label{cdleq2 implies free}
Let $G$ be an orientable $PD_3$-group. 
If the restriction from $H^3(\widehat{G}_p;\mathbb{F}_p)$ to 
$H^3(G;\mathbb{F}_p)$ is trivial then $\widehat{G}_p$ is a free pro-$p$ group.
In particular,  $cd_p(\widehat{G}_p)\not=2$,
and so $\widehat{G}_p$ cannot be a Demu\v skin group.
\end{theorem}

\begin{proof}
Let $j:G\to\widehat{G}_p$ be the canonical homomorphism.
Then  $H_1(j;\mathbb{F}_p)$ and $H^1(j;\mathbb{F}_p)$ are isomorphisms,
while $H_2(j;\mathbb{F}_p)$ is an epimorphism and 
$H^2(j;\mathbb{F}_p)$ is a monomorphism, for any group $G$
If $\gamma\in{H^2(\widehat{G}_p;\mathbb{F}_p)}$ is non-zero
then there is an $\alpha\in{H^1(\widehat{G}_p;\mathbb{F}_p)}$ such that 
$j^*(\alpha\cup\gamma)=j*\alpha\cup{j^*\gamma}\not=0$, 
by the non-degeneracy of Poincaré duality for $G$.
Hence if $H^3(j)=0$ then $H^2(\widehat{G}_p;\mathbb{F}_p)=0$,
and so $\widehat{G}_p$ is a free pro-$p$ group  \cite[Prop. 21]{Serre}.
\end{proof}

Proofs of much of the following corollary can be found in \cite{Desi-Pavel},
but the arguments here differ in some respects.
Condition (3) is closely related to one of the hypotheses in
\cite[Theorem 3.1]{Desi-Pavel}, 
while (2) and the implication $(4)\Rightarrow(1)$ appear to be new.
We recall  that 
$H^i(\widehat{G}_p;\mathbb{F}_p)\cong\varinjlim{H^i(G/U;\mathbb{F}_p)}$, 
for all $i$, the limit being taken over the directed system of normal subgroups $U$ of $p$-power index in $G$.

\begin{corollary}
Let $G$ be an orientable $PD_3$-group and $p$ be a prime.
Then the following are equivalent.
\begin{enumerate}
\item$\widehat{G}_p$ is a pro-$p$ $PD_3$-group;

\item$G$ has a normal subgroup $U$ of $p$-power index such that
inflation from $H^2(G/U;\mathbb{F}_p)$ to $H^2(G;\mathbb{F}_p)$ 
is an epimorphism, 
and each such $U$ has a proper subgroup $V<U$ which is normal 
and of $p$-power index in $G$ and such that inflation from
$H^3(G/U;\mathbb{F}_p)$ to $H^3(G/V;\mathbb{F}_p)$ has rank $1$;

\item{}every subgroup $U<G$ of $p$-power index has 
a proper subgroup $V<U$ of $p$-power index which is normal in $U$
and such that inflation from
$H^2(U/V;\mathbb{F}_p)$ to $H^2(U;\mathbb{F}_p)$ is an epimorphism;

\item$\widehat{G}_p$ has cohomological $p$-dimension $3$ and 
$\chi(\widehat{G}_p)=0$.
\end{enumerate}
\end{corollary}

\begin{proof}
Let $j:G\to\widehat{G}_p$ be the canonical homomorphism.
 If $\widehat{G}_p$ is a pro-$p$ $PD_3$-group  then 
\[
\beta_2(\widehat{G}_p;\mathbb{F}_p)=\beta_1(\widehat{G}_p;\mathbb{F}_p)
=\beta_1(G;\mathbb{F}_p)=\beta_2(G;\mathbb{F}_p),
\]
and so $H^2(j;\mathbb{F}_p)$ is an isomorphism.
Therefore $G$ has a normal subgroup $U$ such that $[G:U]$ 
is a power of $p$,
$H^1(G/U;\mathbb{F}_p)\cong{H^1(G;\mathbb{F}_p)}$ 
and inflation from $H^2(G/U;\mathbb{F}_p)$ to $H^2(G;\mathbb{F}_p)$ 
is onto.
Since $H^1(\widehat{G}_p;\mathbb{F}_p)\not=0$ it follows 
that inflation from $H^3(G/U;\mathbb{F}_p)$ to $H^3(G;\mathbb{F}_p)$
is  an epimorphism, by the non-degeneracy of Poincar\'e duality for $G$.
Since $\varinjlim{H^3}(G/U;\mathbb{F}_p)=
H^3(\widehat{G}_p;\mathbb{F}_p)\cong\mathbb{F}_p$,
there is in turn a subgroup $V<U$ which is normal 
and of $p$-power index in $G$ and such that inflation from
$H^3(G/U;\mathbb{F}_p)$ to $H^3(G/V;\mathbb{F}_p)$ has rank $1$,
i.e, the image is $\mathbb{F}_p$.
Hence $(1)\Rightarrow(2)$.

Conversely, if these conditions hold then $\widehat{G}_p$ is infinite,
since $G$ has subgroups of unbounded $p$-power index,
and $H^2(j;\mathbb{F}_p)$ is an isomorphism.
(In particular, $\widehat{G}_p$ is not a free pro-$p$ group.)
Moreover,  
there is a sequence $U_{i+1}<U_i$ of normal subgroups of $p$-power index such that the inflation from $H^2(G/U_i;\mathbb{F}_p)$ to 
$H^2G;\mathbb{F}_p)$ is an epimorphism and the inflation
from $H^3(G/U_i;\mathbb{F}_p)$ to $H^3(G/U_{i+1};\mathbb{F}_p)$
has rank 1.
This together with Theorem \ref{cdleq2 implies free} implies that
$H^2(j;\mathbb{F}_p)$ is an isomorphism and that
$H^3(\widehat{G}_p;\mathbb{F}_p)\cong\mathbb{F}_p$.
Hence $H^3(j;\mathbb{F}_p)$ is also an isomorphism, by the Theorem,
and so $(2)\Rightarrow(1)$, by \cite[Sect. 4.5, Prop.  32]{Serre},
together with Lemma \ref{obvious}.

A similar argument applies for each subgroup of finite index in $G$, 
since such subgroups are also orientable $PD_3$-groups.
Hence $(1)\Rightarrow(3)$.

If (3) holds and $j_U : U \to \widehat{U}_p$ is the canonical map
then  $H^2(j_U; \mathbb{F}_p)$ is an isomorphism,
so $\beta_2(\widehat{G}_p;\mathbb{F}_p)=\beta_2(G;\mathbb{F}_p)$.
Hence $H_2(j_U; \mathbb{F}_p)$ is also an isomorphism,
and so $(3)\Rightarrow(4)$, by \cite[Theorem 3.1]{Desi-Pavel}.

If (4) holds then $H^3(j;\mathbb{F}_p)\not=0$, 
since $\widehat{G}_p$ has cohomological $p$-dimension $>1$,
by Theorem \ref{cdleq2 implies free}.
Since $\beta_3(\widehat{G}_p;\mathbb{F}_p)\geq1=\beta_3(G;\mathbb{F}_p)$
and $\chi(\widehat{G}_p)=\chi(G)=0$,
we have 
\[
\beta_2(\widehat{G}_p;\mathbb{F}_p)\geq
 \beta_1(\widehat{G}_p;\mathbb{F}_p) = \beta_1({G};\mathbb{F}_p) = 
\beta_2(G;\mathbb{F}_p).
\]
On the other hand $H_2(j; \mathbb{F}_p)$ is surjective, 
hence $\beta_2(\widehat{G}_p;\mathbb{F}_p) \leq \beta_2({G};\mathbb{F}_p)$.
Hence $\beta_2(\widehat{G}_p;\mathbb{F}_p)=
\beta_2(G;\mathbb{F}_p)$ and
$\beta_3(\widehat{G}_p;\mathbb{F}_p)=
\beta_3(G;\mathbb{F}_p) = 1$, and so
 $H_*(j;\mathbb{F}_p)$ and $H^*(j;\mathbb{F}_p)$
are isomorphisms in all degrees. 
It then follows from Lemma \ref{obvious} that $\widehat{G}_p$ 
is a pro-$p$ $PD_3$-group. 
(See also \cite[Prop. 3.2]{W}.) 
Thus $(4)\Rightarrow(1)$.
\end{proof}

We note that whether there is a $PD_3$-group $G$ with 
$3<cd_p(\widehat{G}_p)<\infty$ remains open.

We remark finally that if $G$ is finitely generated
then the order of the torsion subgroup of $G/G'$ is divisible 
by only finitely many primes.
Hence if $G$ is an orientable $PD_3$-group then $\widehat{G}_p$
is finite for all primes $p$ if and only if either $G/G'$ is finite cyclic or $G/G'$ 
is the direct sum of a finite cyclic group with a cyclic 2-group,
and  the 2-lower central series of $G$ terminates after finitely many steps. 
If $\widehat{G}_p$  is of type (b) or (c) for all primes $p$ then 
$\beta_1(G;\mathbb{Q})\geq2$.
Every pro-$p$ completion of $G$ is of type (d)
if and only if $G/G'\cong\mathbb{Z}$.

\section{examples illustrating theorem \ref{PD3-}}
\label{examples}

In this section we shall gives examples of aspherical 3-manifolds whose fundamental groups represent each of the four cases of Theorem \ref{PD3-}.

Examples with $\widehat{G}_p$ cyclic are easily found.
If $M$ is an aspherical  Seifert fibred homology 3-sphere then 
it admits a natural $S^1$-action with finitely many exceptional orbits
with nontrivial finite isotropy subgroups.
If $n$ is prime to the orders of these isotropy subgroups
then the subgroup of $n$th roots of unity in $S^1$ acts freely on $M$, 
with quotient $\overline{M}$, say.
Hence $G=\pi_1(\overline{M})$ is an orientable $PD_3$-group with 
perfect commutator subgroup $G'=G''$ and
$G/G'\cong\mathbb{Z}/n\mathbb{Z}$.
In particular, if $n=p^k$ for some prime $p$ and $k\geq1$ then
 $\widehat{G}_p\cong\mathbb{Z}/p^k\mathbb{Z}$.
(For example, 
if $q,r,s$ are pairwise relative prime and $\frac1q+\frac1r+\frac1s<1$ 
then the Brieskorn manifold
$M(q,r,s)$ is an aspherical $\mathbb{Z}$-homology 3-sphere,
and we may take $p$ relatively prime to $qrs$.)

Since the quaternionic groups $Q(2^n)$ act freely on $S^3$
(for all $n\geq3$),
the Dehn surgery argument of \cite[Theorem 2.6]{CL00} 
may be used to show that these groups act freely on hyperbolic
$\mathbb{Q}$-homology 3-spheres.
By taking the Dehn surgery slope to be a large enough odd number
we may ensure that $H_1(M;\mathbb{Z})$ has odd order,
where $M$ is the resulting $\mathbb{Q}$-homology 3-sphere.
The quotient $\overline{M}=M/(Q(2^n))$ is then an aspherical orientable
3-manifold,
and $G=\pi_1(\overline{M})$ is an orientable $PD_3$-group
with $\widehat{G}_2\cong{Q(2^n)}$.
However we do not have explicit examples of this type.

The simplest example of case (b) of Theorem \ref{PD3-} is $G=\mathbb{Z}^3$,
the fundamental group of the 3-torus.
More generally,
every finitely generated, torsion free nilpotent group
is residually a finite $p$-group for all $p$,
by Theorem 4 of \cite[Chapter 1]{Se}.
Thus the pro-$p$ completion of a nilpotent $PD_3$-group
is a pro-$p$ Poincaré duality group of dimension 3.
(This does not extend to the virtually nilpotent case.
The group $G=\pi_1(M(3_1))$ mentioned below is virtually $\mathbb{Z}^3$,
but $\widehat{G}_p\cong\widehat{\mathbb{Z}}_p$,
for all primes $p$.)

Examples of aspherical 3-manifolds whose fundamental groups illustrate 
cases (c) and (d) may be constructed by surgery on links.
Let $M(L)$ be the closed orientable 3-manifold obtained by 0-framed surgery on the components of an $m$-component link $L$ in $S^3$.
The fundamental group $\pi_1(M(L))$ is the quotient of the link group
$\pi{L}$ by the normal subgroup generated by the longitudes of $L$.
The inclusion of a set of meridians determines a homomorphism 
from the free group $F(m)$ to the link group $\pi{L}$
which induces an isomorphism on abelianization.
If $L$ is a boundary link (in particular, if $m=1$ and so $L$ is a knot)
this homomorphism is split by an epimorphism from $\pi{L}$ to $F(m)$,
and the longitudes of $L$ are in the kernel of any such epimorphism.
The induced homomorphisms between the quotients of the lower central series
$\pi{L}/\gamma_k\pi{L}\to{F(m)}/\gamma_kF(m)$
are isomorphisms, for all $k\geq1$ \cite{St65}.
Hence $J=\pi_1(M(L))$ is an extension of $F(m)$ by 
$\gamma_\omega{J}=\cap_{k\in\mathbb{N}}\gamma_kJ$,
and $\widehat{J}_p\cong\widehat{F(m)}_p$, for all primes $p$.
In our first such example we shall show that $M(L)$ is aspherical;
in the second we show that $M(L)$ must have an aspherical summand
with the requisite properties.

For the first such example we shall let $L$ be the link obtained by 
replacing each component of the Hopf link $2^2_1$ by an untwisted 
Whitehead double \cite[Figure 1.6]{HiAIL}.
(There is a choice involved,  but that is irrelevant for our purposes.)
This is a boundary link, since each component of $L$ bounds a punctured torus
inside a tubular neighbourhood of the corresponding component of the Hopf link.

The components of $L$ are separated by a torus $T\subset{S^3}$.
Each component of $S^3\setminus{T}$ is homeomorphic to $X(Wh)$,
the exterior of the Whitehead link $Wh=5^2_1$.
(The notation $5^2_1$ refers to the tables of knots and links in \cite{Rol}.)
Then $M(L)\cong{N\cup_fN}$, where $f$ is a homeomorphism between the boundaries of the two copies of the 3-manifold $N$ obtained by
attaching a solid torus to the boundary of $X(Wh)$
so that $\partial{D^2}$ is a longitude.
We shall show that $N$ is aspherical and the inclusion of $\pi_1(\partial{N}$
into $\nu=\pi_1(N)$ is injective.
Hence $M(L)$ is aspherical and so $G=\pi_1(M(L))$ is a $PD_3$-group.

The link group $\pi{Wh}$ has a presentation
\[
\langle{a,b,w,x,y}\mid{axa^{-1}=bwb^{-1}=y},~waw^{-1}=xax^{-1}=b,~yxy^{-1}=w\rangle,
\]
and the longitudes for $a$ and $x$ are represented by $\lambda_a=x^{-1}w$ and $\lambda_x=a^{-1}byx^{-1}$, respectively.
We may assume that $N$ is obtained by attaching
$D^2\times{S^1}$ to the component with meridian $x$, 
so that the image of $\lambda_x$ in $\nu=\pi_1(N)$ is trivial.
We have $\lambda_x=a^{-1}bab^{-1}$, since $y=axa^{-1}$.
Hence $\nu$ has the presentation
\[
\langle{a,b,\lambda,x}\mid{axa^{-1}=bx\lambda{b}^{-1}},~
a\lambda=\lambda{a},~xax^{-1}=b, 
\]
\[axa^{-1}xax^{-1}a^{-1}=x\lambda,
~ ab=ba
\rangle.
\]
(Here we have written $\lambda$ for $\lambda_a$ and replaced
$w$ by $x\lambda$ and 
$y$ by $ axa^{-1}$.)
This presentation simplifies to
\[
\langle{a,b,\lambda,x}\mid{a\lambda=\lambda{a}},~ab=ba,~xax^{-1}=b,~
x\lambda{b}^{-1}ax^{-1}=b^{-1}a
\rangle,
\]
since the relation
$\lambda=x^{-1}axa^{-1}xax^{-1}a^{-1}$ follows from the others.
Thus $\nu$ is an HNN extension with base the group 
$\langle{a,b,\lambda}\rangle\cong\mathbb{Z}\times{F(2)}$, 
associated subgroups $\langle{a,\lambda{b}^{-1}}\rangle$
and $\langle{a,b}\rangle$,
and stable letter $x$.
Hence $\nu$ has one end.
The image of $\pi_1(\partial{N})$ is the subgroup $\langle{a,\lambda}\rangle\cong\mathbb{Z}^2$, and so $\partial{N}$ is incompressible in $N$.
It follows from the exact sequence of $(N,\partial{N})$ with coefficients
$\mathbb{Z}[\nu]$, Poincaré-Lefshetz duality and the facts that
$\nu$ has one end and the components of $\partial{N}$ are aspherical 
that $N$ is aspherical.
(See \cite[Lemma 3.1]{Hi20}.)

In our second example the $PD_3$-group $G$ 
does not map onto a nonabelian free group, 
although the pro-$p$ completions of $G$ are free pro-$p$ groups.
Let $L=L_1\cup{L_2}$ be the 2-component link of \cite[Figure 8.1]{HiAIL}.
The homomorphism from $F(2)$ to $\pi{L}$ determined by a pair of meridians 
induces isomorphisms $F(2)/\gamma_nF(2)\cong\pi{L}/\gamma)n\pi{L}$, 
for all $n\geq1$,
and the longitudes of $L$ lie in $\cap_{n\geq1}\gamma_n\pi{L}$.
Hence $\pi_1(M(L))/\gamma_n\pi_1(M(L))\cong\pi{L})/\gamma_n\pi{L}$,
for all $n\geq1$.
The link $L$ is not an homology boundary link: 
there is no epimorphism from $\pi{L}$ to $F(2)$,
and so $\pi_1(M(L))$ does not map onto $F(2)$.
Thus if $M(L)=\sharp_{i=1}^r{M_i}$ is a factorization of $M(L)$ 
as a connected sum of indecomposables all but one of the summands 
must be homology 3-spheres.
We may assume that $M_1$ is not an homology 3-sphere,
and so $H_1(M_1;\mathbb{Z})\cong\mathbb{Z}^2$.
Hence $M_1$ is aspherical, since it is indecomposable, 
and $\pi_1(M_1)$ is infinite and not virtually $\mathbb{Z}$.
(It is likely that $M(L)$ is itself aspherical, 
but we do not need to know this.)
Thus $G=\pi_1(M_1)$ is an orientable $PD_3$-group.
The natural epimorphism from $\pi_1(M)$ to $G$ induces isomorphisms 
$\pi_1(M(L))/\gamma_n\pi_1(M(L))\cong{G}/\gamma_nG$, for all $n\geq1$, 
since the fundamental groups of the other summands of $M(L)$ are all perfect.
Hence $F(2)/\gamma_nF(2)\cong{G}/\gamma_nG$, for all $n\geq1$.
On passing to the $p$-lower central series and pro-$p$ completion,
we conclude that $\widehat{F(2)}_p\cong\widehat{G}_p$, for all primes $p$.

The fundamental groups of orientable closed 3-manifolds
which fibre over non-orientable aspherical surfaces
give further examples of type (c).
The simplest such are the semidirect products
$G=\mathbb{Z}\rtimes_wC$, where $C$ is a $PD_2$-group 
with orientation character $w:C\to\mathbb{Z}^\times$.
Such groups $G$ are orientable $PD_3$-groups.
If $C$ is orientable then $\widehat{G}_p$
is a pro-$p$ $PD_3$-group,
for all primes $p$.
If $C$ is non-orientable then $\widehat{G}_2$ is again a pro-2 $PD_3$-group,
but if $p$ is odd then $\widehat{G}_p\cong\widehat{C}_p$,
by  Lemma \ref{elementary},
and this is a finitely generated free pro-$p$ group,
by Lemma \ref{nonorPD2}.

We have not yet found any examples of case (c) for which
the pro-$p$ completion is not a free pro-$p$ group.

The simplest examples of case (d)  of Theorem \ref{PD3-}
are semidirect products $G=\mathbb{Z}^2\rtimes_A\mathbb{Z}$, 
where $A\in\mathrm{SL}(2,\mathbb{Z})$.
All such groups are solvable $PD_3$-groups.
If $p$ is a prime such that $(\det(A-I),p)=1$ then 
$ \widehat{G}_p \cong\mathbb{Z}_p$,
but if $p$ divides $\det(A-I)$ then $\widehat{G}_p $
is a pro-$p$ $PD_3$-group.
(Note that if one eigenvalue of $A$ is congruent to 1 {\it mod} $(p)$
then so is the other, since they are mutually inverse.)

We may construct further examples of case (d)  
by 0-framed surgery on knots.
If $K$ is a nontrivial fibred knot then $M(K)$ fibres over $S^1$, 
with fibre $F$ a closed orientable surface of genus $\geq1$.
Taking $K$ to be the trefoil knot $3_1$ or the figure-eight knot $4_1$
gives examples with fibre the torus $T$,
and $\pi_1(M(K))\cong\mathbb{Z}^2\rtimes_A\mathbb{Z}$,
where $A=\left(\smallmatrix1&-1\\1&0\endsmallmatrix\right)$ or
$\left(\smallmatrix2&1\\1&1\endsmallmatrix\right)$,
respectively. 
The knots $K=6_2$ and $6_3$ give examples with fibre of genus 2.
Thus if $K$ is a non-trivial fibred knot then $G=\pi_1(M(K))$ 
is a $PD_3$-group which is an extension of $\mathbb{Z}$ 
by the $PD_2$-group $\phi=\pi_1(F)$, and $G'=\phi$.
Hence $G/G'\cong\mathbb{Z}$,
and so the lower central series for $G$ stabilizes at
$\gamma_nG=\gamma_2G=\phi$.
Therefore $\widehat{G}_p\cong\widehat{\mathbb{Z}}_p$, for all primes $p$,
and so $G$ is in case (d).

Examples of dihedral type for case (d) may also be constructed
in terms of knot theory, 
but require a little more work.
Let $K$ be a knot which is carried onto itself by an orientation-reversing involution $h$ of $S^3$ which also reverses the orientation of $K$.
(Such knots are said to be ``strongly -amphicheiral".)
We may assume that $h(X)=X$, where $X$ is the exterior of $K$.
Suppose also that $h$ has just two fixed points.
Then $Fix(h)\subset{K}$, and so $h$ restricts to a fixed-point free involution 
of $X$ which inverts the generator of $H_1(X;\mathbb{Z})$.

Let $X_1$ and $X_2$ be two copies of $X$, 
with a fixed homeomorphism $j:X_1\to{X}_2$,
and let $DX=X_1\cup_{\partial{X}}X_2$ be the double of $X$ along its boundary,
obtained by setting $x=j(x)$ for all $x\in\partial{X}_1$.
Then $H_3(DX;\mathbb{Z})\cong\mathbb{Z}$
and so $X$ is an orientable closed 3-manifold.
We may define an involution $\phi$ by $\phi(x)=j(h(x))$ for $x\in{X_1}$ 
and $\phi(j(x))=h(x)$ for $j(x)\in{X}_2$.
This involution clearly acts freely on $DX$,
and is orientation-preserving, 
so $M=DX/\langle\phi\rangle$ is a closed orientable 3-manifold.
However $\phi$ inverts the generator of 
$H_1(DX;\mathbb{Z})\cong\mathbb{Z}$,
and so $G=\pi_1(M)$ maps onto $D_\infty$,
with kernel $\pi_1(DX)'$.
The abelianization $\pi_1(DX)'/\pi_1(DX)''$ is annihilated by the 
Alexander polynomial $\Delta_K(t)$.

If $K$ is the unknot then $X\cong{S^1}\times{D^2}$, 
$DX\cong{S^1}\times{S^2}$
and $M\cong\mathbb{RP}^3\#\mathbb{RP}^3$, and so $G\cong{D_\infty}$.
If $K$ is non-trivial then $X$ is aspherical and $\partial{X}\to{X}$ is $\pi_1$-injective, and so $DX$ is aspherical.
Hence $M$ is aspherical .
If $\Delta_K(t)\equiv1~mod~(2)$ then 
$\pi_1(DX)'/\pi_1(DX)''$ is a torsion abelian group of odd exponent.
It then follows easily that $\widehat{G}_2\cong\widehat{D_\infty}_2$.
The simplest example of such a knot is $8_3$,
which has  Alexander polynomial $4t^2-9t+4$.

\section{Goodness and Theorem B}
\label{goodA}

In this section we shall consider the profinite completion, 
as well as pro-$p$ completions.
We call $G$ homologically good if for every finite $G$-module $M$ 
the map $H_i ({G},M) \to H_i(\widehat{G},M)$, 
induced by the canonical map  $G \to \widehat{G}$, is an isomorphism.
   
\begin{lemma} \label{LL2} Let $G$ be an abstract group of type $FP_{\infty}$ and  $\mathcal{T}$ be a directed set of finite index normal subgroups in $G$ that defines the profinite topology of $G$. Then the following conditions are equivalent :

a) $G$ is homologically good;

b)   for every finite $G$-module $M$ we have that $  \underset{ U \in {\mathcal T}}{\varprojlim} H_i({U}, M) = 0$  for $i \geq 1$;

c)  for every prime $p$ for  the trivial   $G$-module $\mathbb{F}_p$ we have that $  \underset{ U \in {\mathcal T}}{\varprojlim} H_i({U}, \mathbb{F}_p) = 0$ for $i \geq 1$;

d) $G$ is good.
\end{lemma}

\begin{proof} a) implies b)  Suppose first that $G$ is homologically good. Then
 $$  \underset{ U \in {\mathcal T}}{\varprojlim} H_i({U}, M) \simeq  \underset{ U \in {\mathcal T}}{\varprojlim} H_i(\widehat{U}, M) \simeq H_i( \underset{ U \in {\mathcal T}}{\varprojlim} \widehat{U}, M)=  H_i( \cap_{ U \in {\mathcal T}} \widehat{U}, M)=H_i(1, M) = 0$$
 
 b) implies c) is obvious.
 
 c) implies a) and d)  Let $M$ be a finite $G$-module. By substituting $U$ with a subgroup of finite index we can assume that $U$ acts trivially on $M$. Then by decomposing $M$ as a direct sum of its $p$-primary components we can assume that $M$ is $p$-primary for $p$ prime.
 
 Let $\mathcal{R}$ be a projective resolution of the trivial $\mathbb{Z} G$-module  $\mathbb{Z}$, where all projectives are finitely generated and since by \cite[Thm. 2.5]{Desi-Pavel} after moving from right to left modules  $Tor_i^{\mathbb{Z} G}(\mathbb{Z}_p[[\widehat{G}]], \mathbb{Z}) = 0$ for $i \geq 1$   , we obtain that  $\widehat{\mathcal{R}} = \mathbb{Z}_p[[\widehat{G}]] \otimes_{\mathbb{Z} G} \mathcal{R}$ is exact, hence is a projective resolution of the trivial pro-$p$ $\mathbb{Z}_p[[\widehat{G}]]$-module $\mathbb{Z}_p$.
  Note that for every finite $p$-primary $G$-module $M$ we have that $Hom_{\mathbb{Z} G} ({\mathcal{R}}^{del}, M) \simeq Hom_{\mathbb{Z}_p[[\widehat{G}]]} (\widehat{\mathcal{R}}^{del}, M)$ and $M \otimes_{\mathbb{Z} G} \mathcal{R}^{del} \simeq M \otimes_{\mathbb{Z}_p[[\widehat{G}]]} \widehat{\mathcal{R}}^{del}$,  where  ${\mathcal{R}^{ del}}$, $\widehat{\mathcal{R}}^{del}$  denote the deleted complexes obtained from $\mathcal{R}$ and $\widehat{\mathcal{R}}$ i.e. we substitute the modules $\mathbb{Z}$ and ${\mathbb{Z}}_p$ that are in dimension $-1$ with the zero module. Hence
  $$H^i (G, M) \simeq H^i(Hom_{\mathbb{Z} G} ({\mathcal{R}}^{del}, M)) \simeq H^i( Hom_{\mathbb{Z}_p[[\widehat{G}]]} (\widehat{\mathcal{R}}^{del}, M)) \simeq H^i(\widehat{G}, M)$$
  and the composition of the above isomorphisms is the map $H^i(\widehat{G}, M) \to H^i (G , M)$ induced by the canonical map $G \to \widehat{G}$. Thus  $G$ is good. 
  
Similarly
  $$H_i(G, M) \simeq H_i( M \otimes_{\mathbb{Z} G} \mathcal{R}^{del}) \simeq H_i(M \otimes_{\mathbb{Z}_p[[\widehat{G}]]} \widehat{\mathcal{R}}^{del}) \simeq H_i(\widehat{G}, M)$$
   and the composition of the above isomorphisms is the map $H_i(G, M) \to H_i ( \widehat{G}, M)$ induced by the canonical map $G \to \widehat{G}$. Thus $G$ is homologically good.
   
   d) implies c) Fix a prime $p$ and consider the Pontrygin duality given by $M^* = Hom_{\mathbb{Z}_p} (M, \mathbb{Q}_p/ \mathbb{Z}_p)$ that induces a functorial isomorphism $H^i(G, M^*) \simeq H_i(G, M)^*$ for a finite $G$-module $M$.
Then
$$( \underset{ U \in {\mathcal T}}{\varprojlim} H_i({U}, \mathbb{F}_p))^* \simeq  \underset{ U \in {\mathcal T}}{\varinjlim} H_i({U}, \mathbb{F}_p)^* \simeq
 \underset{ U \in {\mathcal T}}{\varinjlim} H^i({U}, \mathbb{F}_p^*) =   \underset{ U \in {\mathcal T}}{\varinjlim} H^i({U}, \mathbb{F}_p) = 0$$
 where the last equality follows from \cite[Ch. 1, sec. 2, ex.1a)]{Serre}.
\end{proof}

 A group $B$ is called homologically $p$-good 
if for every finite pro-$p$ $\mathbb{Z}_p[[\widehat{B}_p]]$-module $M$ we have that the canonical map $B \to \widehat{B}_p$ induces an isomorphism $H_i({B}, M) \to H_i(\widehat{B}_p, M)$.

\begin{lemma} \label{LL22} Let $G$ be a abstract group of type $FP_{\infty}$, $p$ be a fixed prime number and  $\mathcal{T}$ be a directed set of $p$-power index normal subgroups in $G$ that defines the pro-$p$ topology of $G$. Then the following conditions are equivalent :

a) $G$ is homologically $p$-good; 

b)   for every finite pro-$p$ $\mathbb{Z}_p[[\widehat{G}_p]]$-module $M$ we have that $  \underset{ U \in {\mathcal T}}{\varprojlim} H_i({U}, M) = 0$  for $i \geq 1$;

c)  for  the trivial   $G$-module $\mathbb{F}_p$ we have that $  \underset{ U \in {\mathcal T}}{\varprojlim} H_i({U}, \mathbb{F}_p) = 0$  for $i \geq 1$;

d) $G$ is $p$-good.
\end{lemma}
\begin{proof}
The proof is an obvious modification of the proof of Lemma \ref{LL2}.
\end{proof}

\begin{lemma}  \label{p-good} Let $ 1 \to A \to B \to C \to 1$ be a short exact sequence of abstract groups such that both $A$ and $C$ are $p$-good, $H^i(A,M)$ is finite for any finite pro-$p$ 
$\mathbb{F}_p[[\widehat{B}_p]]$-module $M$
and $ 1 \to \widehat{A}_p \to \widehat{B}_p \to \widehat{C}_p \to 1$ is exact.  Then $B$ is $p$-good. \end{lemma} 

\begin{proof} Consider the LHS spectral sequence $\widehat{E}_{i,j}^2 = H^i(\widehat{C}_p, H^j(\widehat{A}_p, M))$ that converges to $H^{i+j} (\widehat{B}_p, M)$, where $M$ is a finite pro-$p$ $\mathbb{Z}_p[[\widehat{B}_p]]$-module.
Consider the LHS spectral sequence ${E}_{i,j}^2 = H ^i({C}, H ^j({A}, M))$ that converges to $H^{i+j} ({B}, M)$. By the $p$-goodness of $A$ and $C$ the map $B \to \widehat{B}_p$ induces an isomorphism $ \widehat{E}_{i,j}^2 \to E_{i,j}^2$. By the naturality of the spectral sequence we conclude by induction on $k \geq 2$ that  the map $B \to \widehat{B}_p$ induces an isomorphism $\widehat{E}_{i,j}^k \to E_{i,j}^k $, hence an isomorphism  $\widehat{E}_{i,j}^{\infty} \to E_{i,j}^{\infty}$. Then the convergence of the spectral sequence implies that  the map $B \to \widehat{B}_p$ induces an isomorphism $ H^{i+j}(\widehat{B}_p, M) \to H^{i+j} (B, M)$.
\end{proof}

\begin{lemma} \label{LL1}
Any orientable surface group is good and $p$-good.
\end{lemma}

\begin{proof}The goodness is a particular case of \cite[Thm. 1.3]{G-J-Z} and the $p$-goodness is a particular case of \cite[Thm. A]{D0}. Alternatively both statements have elementary proofs using the results from the previous and this section. 
\end{proof}

In particular, the pro-$p$ completion of an orientable $PD_2$-group 
is a pro-$p$ $PD_2$-group.
The situation is somewhat different in the non-orientable case.

\begin{lemma}
\label{nonorPD2}
Let $C$ be a non-orientable $PD_2$-group.
Then $\widehat{C}_2$ is a pro-$2$ $PD_2$-group
but $\widehat{C}_p$ is a free pro-$p$ group, for every odd prime $p$.
\end{lemma}

\begin{proof}
Let $C^+$ be the kernel of the orientation character
$w:C\to\mathbb{Z}^\times$.
Then  $[C:C^+]=2$, since $C$ is non-orientable, 
and so $\widehat{C^+}_2$ has index 2 in $\widehat{C}_2$.
Since $C^+$ is an orientable $PD_2$-group 
it follows that $\widehat{C}_2$ is a pro-$2$ Poincaré  duality group of dimension $2$.

Assume now that $p$ is an odd prime,.
Then $H_1(C;\mathbb{F}_p)\cong\mathbb{F}_p^r$, for some $r\geq1$,
and $H_2(C;\mathbb{F}_p)=0$.
Let $F$ be the free group of rank $r$ and $f:F\to{C}$ a homomorphism 
such that $H_1(f;\mathbb{F}_p)$ is an isomorphism.
Since $H_2(f;\mathbb{F}_p)$ is also an isomorphism,
$f$ induces isomorphisms on all corresponding quotients 
of the $p$-lower central series of these groups \cite{St65}.
Hence $\widehat{F}_p\cong\widehat{C}_p$.
\end{proof}

We return briefly to consider profinite completion, 
rather than pro-$p$ completions.

\begin{proposition} \label{profinite}
Let $1 \to A \to B \to C \to 1$ be a short exact sequence of groups such that 
$A$  is an orientable surface group,  
$B$ is an orientable $PD_{2 + m}$-group and $C$  is a good $PD_m$ group.
Then $\widehat{B}$ is a strong  orientable profinite $PD_{2+m}$-group at $p$. 
In particular, if $C$ is an orientable surface group (so $m = 2$) then $\widehat{B}$ is a strong orientable profinite $PD_4$-group.
\end{proposition}  

\begin{proof} By Lemma \ref{LL1} surface groups are good and by
\cite[Ch. 1, Sec. 2.6, Ex. 2c)]{Serre} so are extensions of good groups 
where the bottom group is $FP_{\infty}$.
In particular, $B$ is good. 
Let $\mathcal{T}$ be a directed set of normal subgroups 
of finite index in $B$ that defines the profinite topology of $B$. 
Then by Lemma \ref{LL2}
 $ \underset{ U \in {\mathcal T}}{\varprojlim} H_i(U, \mathbb{F}_p) = 0 \hbox{ for } i \geq 1.$ Then we can apply Theorem \ref{PD-PD}.
\end{proof}

There is a subtle point here; 
a ``good" group need not be $p$-good for any prime $p$.
The simplest example is perhaps the group $\pi_1(M(3_1))$ of 
\S\ref{proPD3} mentioned above.
There is an exact sequence $1\to{A}\to{B}\to{C}\to1$
with $A\cong\mathbb{Z}$, $B=\pi_1(M(3_1))$ and $C\cong\mathbb{Z}$,
and so $B$ is good, by Proposition \ref{profinite}.
However, $\widehat{B}_p\cong\widehat{C}_p=\widehat{\mathbb{Z}}_p$, 
and so $B$ is not $p$-good, for any prime $p$.

We may now prove Theorem B.

\begin{theorem}
\label{Theorem B} 
Let $1 \to A \to B \to C \to 1$ be a short exact sequence of groups such that 
$A$ is  an orientable surface group, $B$ is an orientable $PD_{s+2}$-group
and $C$ is an orientable $PD_s$ group, where $s = 2$ or $s = 3$.  
Let $\mathcal{T}$ be a directed set of  normal subgroups of $p$-power index 
in $B$ that defines the pro-$p$ topology of $B$. 
Suppose that 
$  \underset{ U \in {\mathcal T}}{\varprojlim} H_1(U \cap A, \mathbb{F}_p) = 0$ and furthermore if $s = 3$, there is an upper bound on the deficiency of the subgroups of finite index in $\widehat{C}_p$. Then 
    
    a) if $s = 2$, then  $\widehat{B}_p$ is a pro-$p$ $PD_4$ group and $B$ is $p$-good. 
    
    b) if $s = 3$,  then $\widehat{B}_p$ is virtually a pro-$p$ $PD_k$-group for some $k \in \{ 2,3, 5\}$. If $k = 5$ then $B$ and $C$ are $p$-good.
    
If additionally $B$ is orientable and $p$-good then $\widehat{B}_p$ is an orientable pro-$p$ $PD_{2+ s}$ group.
\end{theorem}

\begin{proof}
Note that $B$ is a  $PD_{s+ 2}$-group.  Let  $\overline{A}$ be closure of the image of $A$ in $\widehat{B}_p$ i.e. $\overline{A}$ is the completion of $A$ with respect $\{ U \cap A \ | \ U \in \mathcal{T} \}$. Then we have a short exact sequence of pro-$p$ groups
$$1 \to \overline{A} \to \widehat{B}_p \to \widehat{C}_p \to 1.$$
By Lemma \ref{pro-p} 
we see that
 $\overline{A} \simeq \widehat{A}_p$ 
is an orientable pro-$p$ $PD_2$-group.
  
  If $s = 2$ then $\widehat{C}_p$ is an orientable pro-$p$ $PD_2$-group. Hence $\widehat{B}_p$ is a pro-$p$ $PD_4$-group.
  
  If $s = 3$ and $\widehat{C}_p$ is infinite then by the remark after Theorem \ref{PD3-} we have two options :   $\underset{ U \in {\mathcal T}}{\varprojlim} H_2(V_U, \mathbb{F}_p) = 0$ or  $\underset{ U \in {\mathcal T}}{\varprojlim} H_2(V_U, \mathbb{F}_p) =  \mathbb{F}_p$.
In the former $\widehat{C}_p$ is an orientable pro-$p$ $PD_3$-group and in the latter $\widehat{C}_p$ is virtually $\mathbb{Z}_p$. Since $\overline{A} = \widehat{A}_p$ is a pro-$p$ $PD_2$-group we conclude that if  $\widehat{C}_p$ is infinite then $\widehat{B}_p$ is a pro-$p$ $PD_5$-group or virtually a pro-$p$ $PD_3$-group.

The $p$-goodness follows from Lemma \ref{p-good}. We need  that in the case $s = 3$,  if $\widehat{C}_p$ is an orientable pro-$p$ $PD_3$-group then  $C$ is $p$-good, that follows from \cite[Thm. A]{Desi-Pavel}.

If $B$ is $p$-good and $B$ is orientable then by Theorem \ref{PD-PD} $\widehat{B}_p$ is orientable pro-$p$ $PD_{2+ s}$ group. 
\end{proof}

 \begin{lemma} \label{orientation} 
Let $ 1 \to K \to G \to D \to 1$ be a short exact sequence of groups, 
where $G$ is a $PD_n$-group, $D$ is a $PD_{n-1}$-group and $K \simeq \mathbb{Z}$.  Then $G$ is an orientable $PD_n$-group if and only if $K$ is the dualizing module of $D$. 
 \end{lemma}
 
 \begin{proof}
 Consider the LHS spectral sequence $E_{i,j}^2 = H_i(D, H_j(K, \mathbb{Z}))$ that converges to $H_{i+ j}(G, \mathbb{Z})$. Since $E_{i,j}^2 = 0$ if $i \geq n$ or $j \geq 2$, by the convergence  we conclude that $H_n(G, \mathbb{Z}) \simeq H_{n-1}(D, H_1(K, \mathbb{Z}))$.  Let $W $ be the dualizing module of $D$ and $V$ be the dualizing module of $G$. Both $V$ and $W$ are infinite cyclic as abelian groups but in general the corresponding actions of $D$ and $G$ need not be trivial. Then $$H^0(G, V) \simeq H_n(G, \mathbb{Z}) \simeq   H_{n-1}(D, H_1(K, \mathbb{Z})) \simeq H_{n-1}(D, K) \simeq H^0 (D, K \otimes W).$$
Since $V$ and $K \otimes W$ are infinite cyclic as abelian groups, we have that $G$ is orientable $\iff$ $G$ acts trivially on $V$ $\iff$ $H^0(G, V) \not= 0$ $\iff$ $ H^0 (D, K \otimes W) \not= 0$ $\iff$ $K \otimes W$ is the trivial $D$-module (via the diagonal action) $\iff$ $K \simeq W$ as $D$-module.
 \end{proof}
 
\begin{lemma} \label{elementary} 
Suppose $p \not= 2$ and $ 1 \to K \to G \to C \to 1$ is a short exact sequence 
of groups, where the action of $C$ via conjugation on $K = \mathbb{Z}$ 
is non-trivial. 
Then $\widehat{G}_p \simeq \widehat{C}_p$. 
\end{lemma}

\begin{proof}
Suppose that $U$ is a normal subgroup of $p$-power index in $G$. Then for some $i$ we have that $K^{2^i} = [K, G, \ldots, G] \subseteq  \gamma_{i+1} (G) \subseteq U$, hence $K^{2^i} \subseteq U \cap K \subseteq K$. Since $[K :  U \cap K]$ is a $p$-power that divides $2^i = [K : K^{2^i}]$,  we conclude that $K = U \cap K \subseteq U$. Hence  $\widehat{G}_p \simeq \widehat{C}_p$. 
\end{proof}

\section{$PD_4$-groups with Euler characteristic 0}
\label{PD4-0}

In this section we shall prove Theorem C of the Introduction.
(This is Theorem \ref{thmB} below.)

\begin{lemma}
\label{chi>1}
Let $G$ be an orientable $PD_4$-group. 
If the pro-$p$ completion $\widehat{G}_p$ is a pro-$p$ $PD_3$-group
then $\chi(G)\geq2$.
\end{lemma}

\begin{proof}
We have
$\beta_1(\widehat{G}_p;\mathbb{F}_p)=\beta_1(G;\mathbb{F}_p)$ and 
$\beta_2(\widehat{G}_p;\mathbb{F}_p)\leq\beta_2(G;\mathbb{F}_p)$,
for any group $G$
\cite[\S2.6]{Serre}.
Since $\widehat{G}_p$ is a pro-$p$ $PD_3$-group, 
$\beta_1(\widehat{G}_p;\mathbb{F}_p)=
\beta_2(\widehat{G}_p;\mathbb{F}_p)$,
and since the image of $H^2(\widehat{G}_p;\mathbb{F}_p)$
in $H^2(G;\mathbb{F}_p)$ is self-annihilating under cup product,
$\beta_2(G;\mathbb{F}_p)\geq2\beta_1(G;\mathbb{F}_p)$,
by the non-singularity of Poincaré duality.
Hence $\chi(G) =  \sum_{0 \leq i \leq 4} (-1)^i \beta_i(G; \mathbb{F}_p) = 2 - 2 \beta_1(G; \mathbb{F}_p) + \beta_2(G; \mathbb{F}_p) \geq 2$.
\end{proof}

Let $G$ be a group with a normal subgroup $A\cong\mathbb{Z}^2$.
If we fix a basis for $A$ we may identify $\mathrm{Aut}(A)$ with
$\mathrm{GL}(2,\mathbb{Z})$, and conjugation in $G$ then
determines an action $\theta:G/A\to\mathrm{GL}(2,\mathbb{Z})$.
We shall say that the action is orientable if its image lies in
$\mathrm{SL}(2,\mathbb{Z})$.
(Thus $G$ acts orientably if and only if the induced action by
$\det\theta$ on $A\wedge{A}\cong\mathbb{Z}$ is trivial.)

\begin{theorem}
\label{thmB}
Let $1 \to A \to B \to C \to 1$ be a short exact sequence of groups  
such that $A \simeq \mathbb{Z}^2$, 
$B$ is an orientable $PD_4$ group and $C$ is an orientable $PD_2$-group. 
Then one of the following holds:

a) $\widehat{B}_p$ is an orientable pro-$p$ $PD_4$-group and 
$B$ is $p$-good;

b) $\widehat{B}_p\cong\widehat{C}_p$,
 and the image of $A$ in $\widehat{B}_p$ is trivial.
\end{theorem}

\begin{proof}
Let $[B,A]$ be the normal subgroup generated by commutators 
$gag^{-1}a^{-1}$ for $g\in{B}$ and $a\in{A}$.
Let $\Gamma^1A=[B,A]+pA$ and 
$\Gamma^{k+1}A=[B,\Gamma^kA]+p^{k+1}A$,
for all $k\geq1$.
These subgroups have finite index in $A$ and are normal in $B$.
The quotient $A/\Gamma^1A$ is central in $B/\Gamma^1A$,
and $B$ acts on the finite $p$-group $A/\Gamma^kA$ 
through a finite $p$-group, for all $k\geq1$.
Hence $B$ has a normal subgroup $U$ of $p$-power index such that
$A<U$ and $A/\Gamma^kA$ is central in $U/\Gamma^kA$.
The quotient $U/A$ is an orientable $PD_2$-group and
$A/\Gamma^kA$ is a finite abelian group of exponent dividing $p^k$.
On applying Lemma \ref{H2} several times, we see that 
the class in $H^2(U/A;A/\Gamma^kA)$ of the central extension
\[
0\to{A/\Gamma^kA}\to{U/\Gamma^kA}\to{U/A}\to1
\]
restricts to 0 in a normal subgroup $V/A$ of $p$-power index.
Hence $V/\Gamma^kA\cong(V/A)\times(A/\Gamma^kA)$.
An argument by induction on nilpotency class shows that 
$\cap_{k\geq1}\Gamma^kA$ has trivial image 
in every finite quotient of $G$ which is a $p$-group.
It follows that $\cap_{k\geq1}\Gamma^kA$ is the kernel of the pro-$p$ completion homomorphism from $B$ to $\widehat{B}_p$.

Fix an isomorphism $A\cong\mathbb{Z}^2$,
and let $\theta:B\to\mathrm{GL}(2,\mathbb{Z})$ be
the action of $B$ an $A$ induced by conjugation in $G$.
Let $\theta_p:B\to\mathrm{GL}(2,\mathbb{F}_p)$ be the
$mod$-$p$ reduction of $\theta$.

If $\theta_p(g)-I$ is not invertible (for some $g\in{B}$)
then $\theta_p(g)$ has 1 as an eigenvalue.
Since $B$ and $C$ are orientable the action of $B$ on $A$ is orientable.
Hence both eigenvalues of $\theta_p(g)$ are 1, 
since they are mutually inverse,
and so $(\theta_p(g)-I)^2=0$.

Suppose first that this holds for all $g\in{B}$.
Then we may assume that
$\theta_p(B)\leq{U(2,\mathbb{F}_p)}$, 
the subgroup of upper unitriangular matrices \cite[8.1.10]{Rob}.
Thus $A$ has a basis $\{e_1,e_2\}$ such that
$[B,e_1]\leq{pA}$ and $[B,e_2]+pA\leq\mathbb{Z}e_1+pA$.
Hence $\Gamma^1A\leq\mathbb{Z}e_1+pA$.
Define subgroups $[B,_s,e_1]$ inductively by setting $[B,_1e_1]=[B,e_1]$ 
and $[B,_{s+1}e_1]=[B,[B,_se_1]]$ for $s\geq1$.
Then by induction on $k$ we have
$\Gamma^kA\leq{p^k}A+\Sigma_{h+j=k-1}\mathbb{Z}p^h[B,_je_1]$,
$[B,_{2k-1}e_1]\leq{p^k}A$ and 
$[B,_{2k}e_1]\leq\mathbb{Z}p^ke_1+p^{k+1}A\leq{p^k}A$.
Thus $\Gamma^2kA\leq{p^k}A$, for all $k\geq1$,
and so $\cap_{k\geq1}\Gamma^kA=1$.
In this case $\overline{A}\cong\widehat{\mathbb{Z}}_p^2$
and so $\widehat{B}_p$ is a pro-$p$ $PD_4$-group.
Since $\beta_1(\widehat{B}_p;\mathbb{F}_p)=\beta_1(B;\mathbb{F}_p)$
and $\chi(\widehat{B}_p)=0=\chi(B)$, 
it follows that
$\beta_2(\widehat{B};\mathbb{F}_p)=\beta_2(B;\mathbb{F}_p)$.
It then follows easily from Lemma \ref{obvious} and the nonsingularity of 
Poincaré duality that $B$ is $p$-good.

If $\theta_p(g)-I$ is invertible in $\mathrm{GL}(2,\mathbb{F}_p)$ 
for some $g\in{C}$ then $A=[B,A]+pA$.
Hence $A=[B,A]+p^kA$ for all $k\geq1$,
by the Burnside Basis theorem \cite[5.3.2]{Rob}
(equivalently,  by Nakayama's Lemma),
applied to the finite $p$-group $A/p^kA$.
Hence $\cap_{k\geq1}\Gamma^kA=A$,
so $\overline{A}=1$ and $\widehat{B}_p\cong\widehat{C}_p$
is a pro-$p$ $PD_2$-group.
\end{proof}

If $C$ is a non-orientable $PD_2$-group then $B$ is orientable if and only if the determinant of the action is the orientation character of $C$).
In this case the above argument goes through with little change for $p=2$.

{\bf Remark}  Suppose $p \not= 2$ and $ 1 \to K \to G \to C \to 1$ be a short exact sequence of groups, where $C$ is an orientable surface group and the action of $C$ via conjugation on $K = \mathbb{Z}$ is non-trivial. Thus $G$ is a non-orientable $PD_3$-group. Consider $B = S \times G$, where $S = \mathbb{Z}$. Then for $A = S \times K$ we have the short exact sequence of groups $ 1 \to A \to B \to C \to 1$ with $A$ and $C$ orientable surface groups, but  by Lemma \ref{orientation}  $B$ is not orientable. Then  by Lemma \ref{elementary}   $\widehat{B}_p \simeq \widehat{S}_p \times \widehat{G}_p = \mathbb{Z}_p \times \widehat{C}_p$ is an orientable pro-$p$ $PD_3$-group and $\overline{A} \simeq \mathbb{Z}_p$. This  is an example of a group that is not orientable and does not satisfy the conclusions of Theorem C but the only assumption of Theorem C it fails is the orientability one.

Taking products of aspherical 3-manifolds exemplifying case (d) 
of Theorem \ref{PD3-} with the circle
gives examples illustrating part (b) of Theorem C. 
Let $M=M(K)$, where $K$ is a non-trivial fibred knot.
Then $E = M\times{S^1}$ is an aspherical 4-manifold with 
fundamental group $G\times\mathbb{Z}$, 
and $E$ fibres over the torus $T$ with fibre $F$.
If $B$ is a surface of genus $h>1$ and $f:B\to{T}$
is a degree-1 map then the total space  $E_f = N \times_T B$ of the
pullback of the fibration  $N \to T$ over $f$ is aspherical,
and $\pi=\pi_1(E_f)$ is an extension of $\rho=\pi_1(B)$ by $\phi$.
It is easy to see that $\phi\leq\gamma_k\pi$ for all $k$,
and so $\widehat\pi_p\cong\widehat{\rho}_p$, for all primes $p$.

If $L$ is the 2-component boundary link obtained by Whitehead doubling 
each component of the Hopf link then $M(L)\times{S^1}$ is 
an aspherical orientable 4-manifold,
but does not fibre over a surface.
The pro-$p$ completion of $G\times\mathbb{Z}$ is 
$\widehat{F(2)}_{ p}\times\widehat{\mathbb{Z}}_p$,
which has cohomological $p$-dimension 2, 
but is not a Demu\v skin group. 

Note that the hypothesis
  ``$ \underset{ U \in {\mathcal T}}{\varprojlim} H_1(U \cap A, \mathbb{F}_p) 
  = 0$" of Theorem B does not hold for these examples.

\section{Dimension drop}
 \label{more}
 
An orientable $PD_n$-group $G$ has dimension drop $k$ on 
pro-$p$ completion if $\widehat{G}_p$ is a pro-$p$ $PD_{n-k}$-group.
There are aspherical closed orientable $n$-manifolds $N$ 
such that $\pi_1(N)$ has dimension drop $k$ (for all primes $p$),
for all $n\geq2$ and $2\leq{k}\leq{n}$, except when $n=k=5$ \cite{H-K-L}.
This exception reflects the fact that 5 is not in the additive semigroup 
generated by 3 and 4,
dimensions in which aspherical homology spheres are known.
We shall fill this gap below.
However, whether there are any examples of dimension drop 1
remains an open question.

Let $X$ be a compact $4$-manifold whose boundary components 
are diffeomorphic to the 3-torus $T^3$.
A {\it Dehn filling\/} of a component $Y$ of $\partial{X}$ is the adjunction of $T^2\times{D^2}$ 
to $X$ via a diffeomorphism $\partial(T^2\times{D^2})\cong{Y}$.
If the interior of $X$ has a complete hyperbolic metric
then ``most" systems of Dehn fillings on some or all of the boundary components 
give manifolds which admit metrics of non-positive curvature, 
and the fundamental groups of the cores of the solid tori $T^2\times{D^2}$ 
map injectively to the fundamental group of the filling of $X$,
by the Gromov-Thurston $2\pi$-Theorem.
(Here ``most" means ``excluding finitely many fillings of each boundary component".
See \cite{And06}.)

\begin{theorem}
\label{perfect5}
There are aspherical closed $5$-manifolds with perfect fundamental group.
\end{theorem}

\begin{proof}
Let $M=S^4\setminus5T^2$ be the complete hyperbolic 4-manifold with 
finite volume and five cusps considered in \cite{Iv04} and \cite{RT05}, 
and let $\overline{M}$ be a compact core, with interior diffeomorphic to $M$.
Then $H_1(\overline{M};\mathbb{Z})\cong\mathbb{Z}^5$, $\chi(\overline{M})=2$ 
and the boundary components of $\overline{M}$ are all diffeomorphic to the 3-torus $T^3$.
There are  infinitely many quintuples of Dehn fillings of the components of $\partial\overline{M}$ 
such that the resulting closed 4-manifold is an aspherical homology 4-sphere \cite{RT05}.
Let $\widehat{M}$ be one such closed 4-manifold, and let $N\subset\widehat{M}$ 
be the compact 4-manifold obtained by leaving one boundary component of $X$ unfilled.
We may assume that the interior of $N$ has a non-positively curved metric,
and so $N$ is aspherical.
The Mayer-Vietoris sequence for $M=N\cup{T^2\times{D^2}}$  gives an isomorphism
\[
H_1(T^3;\mathbb{Z})\cong{H_1(N;\mathbb{Z})}\oplus{H_1(T^2;\mathbb{Z})}.
\]
Let $\{x,y,z\}$ be a basis for $H_1(T^3;\mathbb{Z})$ compatible with this splitting.
Thus $x$ represents a generator of $H_1(N;\mathbb{Z})$ and maps to 0 in the second summand,
while $\{y,z\}$ has image 0 in $H_1(N;\mathbb{Z})$ but generates the second summand.
Since the subgroup generated by $\{y,z\}$ maps injectively to $\pi_1(\widehat{M)}$ \cite{And06},
the inclusion of $\partial{N}$  into $N$ is $\pi_1$-injective.
Let $\phi$ be the automorphism of $\partial{N}=T^3$ 
which swaps the generators $x$ and $y$,
and let $P=N\cup_\phi{N}$.
Then $P$ is aspherical and $\chi(P)=2\chi(N)=4$.
A Mayer-Vietoris calculation gives $H_1(P;\mathbb{Z})=0$, 
and so $\pi=\pi_1(P)$ is perfect and $H^2(P;\mathbb{Z})\cong\mathbb{Z}^2$.

Let $e$ generate a direct summand of 
$H^2(\pi;\mathbb{Z})=H^2(P;\mathbb{Z})$, 
and let $E$ be the total space of the $S^1$-bundle over $P$ 
with Euler class $e$.
Then $E$ is an aspherical 5-manifold, 
and $G=\pi_1(E)$ is the central extension of $\pi_1(P)$ by $\mathbb{Z}$ 
corresponding to $e\in {H^2(\pi_1(P);\mathbb{Z})}$. 
The Gysin sequence for the bundle (with coefficients in $\mathbb{F}_p$) 
has a subsequence
\[
0\to{H^1}(E;\mathbb{F}_p)\to{H^0(P;\mathbb{F}_p)}\to{H^2(P;\mathbb{F}_p)}\to
{H^2(E;\mathbb{F}_p)}\to\dots
\] 
in which the {\it mod}-$p$ reduction of $e$ generates the image of $H^0(P;\mathbb{F}_p)$.
Since $e$ is indivisible this image is nonzero, for all primes $p$.
Therefore $H^1(G;\mathbb{F}_p)=H^1(E;\mathbb{F}_p)=0$, for all $p$, and so $G$ is perfect.
\end{proof}

We may use such groups to complete the results of \cite{H-K-L}.

\begin{theorem}
For each $r\geq0$ and  $n\geq\max\{r+2,3\}$ there is an aspherical closed 
$n$-manifold with fundamental group $\pi$ such that 
$\pi/\pi'\cong\mathbb{Z}^r$ and $\pi'=\pi''$.
\end{theorem}

\begin{proof}
Let $\Sigma$ be an aspherical homology 3-sphere 
(such as the Brieskorn 3-manifold $\Sigma(2,3,7)$)
and let $P$ and $E$ be as in Theorem \ref{perfect5}.
Taking suitable products of copies of $\Sigma$, $P$, $E$  and $S^1$
with each other realizes all the possibilities with $n\geq{r+3}$,
for all $r\geq0$.

Let $M=M(K)$ be the 3-manifold obtained by $0$-framed surgery 
on a nontrivial prime knot $K$ with Alexander polynomial $\Delta(K)=1$ 
(such as the Kinoshita-Terasaka  knot $11_{n42}$).
Then $M$ is aspherical, since $K$ is nontrivial \cite{Ga87}, and
if $\mu=\pi_1(M)$ then $\mu/\mu'\cong\mathbb{Z}$ and $\mu'$ is perfect,
since $\Delta(K)=1$.
Hence products $M\times(S^1)^{r-1}$ give examples with $n=r+2$, for all $r\geq1$.
\end{proof}

In particular, the dimension hypotheses in Theorem 6.3 
of \cite{H-K-L} may be simplified,
so that it now asserts:

{\it
Let  $m\geq3$ and $r\geq0$.
Then there is an aspherical closed $(m+r)$-manifold $M$ with fundamental group 
$G=K\times\mathbb{Z}^r$, where $K=K'$.
If $m\not=4$  we may assume that $\chi(M)=0$, and if $r>0$ this must be so.}

This is best possible, as no $PD_1$- or $PD_2$-group is perfect,
and no perfect $PD_4$-group $H$ has $\chi(H)=0$.

As observed above, there are no known examples of dimension drop 1.
No $PD_n$-group with $n\leq3$ has such a dimension drop 
on any $p$-profinite completion.
(This is clear if $n\leq2$, 
and follows from Theorem \ref{cdleq2 implies free} if $n=3$.)
Hence we may focus on the first undecided case, $n=4$. 

In seeking possible examples of dimension drop 1 
in the pro-$p$ completion of a $PD_n$-group,
the most convenient candidates are groups whose lower central series terminates after finitely many steps.
A finitely generated nilpotent group $\nu$ of Hirsch length $h$ 
has a maximal finite normal subgroup $T(\nu)$, 
with quotient a $PD_h$-group.
Moreover, $\nu/T(\nu)$ has  nilpotency class $<h$,
and is residually a finite $p$-group for all $p$,
by Theorem 4 of \cite[Chapter 1]{Se}.
Thus the pro-$p$ completion of $\nu$ is a pro-$p$ $PD_h$-group.
for all $p$ prime to the order of $T(\nu)$.

If $\gamma_kG/\gamma_{k+1}G$ is finite, of exponent $e$, say,
then so are all subsequent subquotients of the lower central series,
by Proposition 11 of Chapter 1 of \cite{Se}.
Thus if $G$ is a $PD_4$-group such that $G/\gamma_3G$ 
has Hirsch length 3
and $\gamma_3G/\gamma_4G$ is finite then, 
setting $\nu=G/\gamma_3G$,
the canonical projection to $\nu/T(\nu)$ induces isomorphisms on pro-$p$ completions, 
for almost all primes $p$.
Taking products of one such group with copies of $\mathbb{Z}$
would give similar examples with dimension drop 1 in all higher dimensions.

Let $G$ be the fundamental group of a closed orientable 4-manifold
which is the total space of a bundle with base and fibre aspherical closed orientable surfaces.
Thus there is an epimorphism $f:G\to{C}$ with kernel $A$, 
where $A$ and $C$ are orientable $PD_2$-groups.
The projection $f$ induces an epimorphism
$\hat{f}:\widehat{G}_p\to\widehat{C}_p$ of pro-$p$ completions.
Let $K$ be the kernel of the canonical homomorphism from $G$ to 
$\widehat{G}_p$.
The kernel of $\widehat{f}$ is the closure of the image of $A$,
and so is topologically finitely generated.
If $\widehat{G}_p$ is a pro-$p$ $PD_3$-group then
$\mathrm{Ker}(\widehat{f} )\cong\mathbb{Z}_p$ \cite[Cor. 4]{Desi-Aline}.
Hence $A/K$ is finitely generated and abelian of rank 1.
An immediate consequence is that 
$\beta_1(G;\mathbb{F}_p)=\beta_1(C;\mathbb{F}_p)$
or $\beta_1(C;\mathbb{F}_p)+1$.
This condition is not satisfied by most such surface bundle groups $G$,
as $\beta_1(G;\mathbb{F}_p)$ may be as large as
$\beta_1(A;\mathbb{F}_p)+\beta_1(C;\mathbb{F}_p)$.
There are no such bundles with base or fibre the torus,
by Lemma \ref{chi>1}.

We make one further observation, related to Lemma \ref{chi>1}.
If $G$ is an orientable $PD_4$-group and $\widehat{G}_p$ is a pro-$p$ Poincaré duality group  of dimension 3 then the canonical homomorphism from
$H^3(\widehat{G}_p;\mathbb{F}_p)$ to $H^3(G;\mathbb{F}_p)$ is trivial.
For $H^1(\widehat{G}_p;\mathbb{F}_p)\not=0$
and so there are classes 
$\alpha\in{H^1(\widehat{G}_p;\mathbb{F}_p)}=H^1(F;\mathbb{F}_p)$
and $\beta\in{H^2}(\widehat{G}_p;\mathbb{F}_p)<H^2(G;\mathbb{F}_p)$
such that $\alpha\cup\beta$ generates $H^3(\widehat{G}_p;\mathbb{F}_p)$,
by Poincaré duality for $\widehat{G}_p$.
If this has nonzero image in $H^3(G;\mathbb{F}_p)$ then there is a
$\gamma\in{H^1(\widehat{G}_p;\mathbb{F}_p)}=H^1(F;\mathbb{F}_p)$
such that $\alpha\cup\beta\cup\gamma\not=0$ in $H^4(G;\mathbb{F}_p)$.
But this cup product is in the image of $H^4(\widehat{G}_p;\mathbb{F}_p)$,
which is 0.
An equivalent formulation of this condition is that
inflation from $H^3(G/U;\mathbb{F}_p)$ to $H^3(G;\mathbb{F}_p)$ is
trivial for every normal subgroup $U$ of $p$-power index in $G$.
In particular, taking $U=G'X^p(G)$ (where $X^p(G)$ is the verbal subgroup generated by all $p$th powers) we see that the image of 
$\wedge^3H^1(G;\mathbb{F}_p)$ in $H^3(G;\mathbb{F}_p)$ must be 0.

\end{document}